\documentclass[journal]{IEEEtran}
\IEEEoverridecommandlockouts  

\usepackage{placeins}
\usepackage[pdftex]{graphicx}
\usepackage{color}
\graphicspath{{./figs/}{./jpeg/}}
\DeclareGraphicsExtensions{.pdf,.jpeg,.png}

\usepackage[cmex10]{amsmath}
\usepackage{bm,upgreek}
\usepackage{amsfonts}
\usepackage{wasysym}
\usepackage{amsthm}
\usepackage{amssymb}
\usepackage{import}
\usepackage{booktabs}
\theoremstyle{plain}
\newtheorem{theorem}{Theorem}
\newtheorem{proposition}{Proposition}

\newtheorem{corollary}{Corollary}[theorem]
\newtheorem{lemma}{Lemma}
\usepackage{balance}
\theoremstyle{definition}
\newtheorem{problem}{Problem}
\newtheorem{definition}{Definition}

\theoremstyle{remark}
\newtheorem{remark}[theorem]{Remark}

\usepackage{import}

\DeclareMathOperator{\sign}{sign}

\usepackage{algorithmic}

\usepackage{array}

\usepackage{url}
\usepackage{siunitx}
\begin{document}
\newcommand*{\htwo}{\mathcal H_2}
\newcommand*{\hinf}{\mathcal H_\infty}
\newcommand*{\jo}{(j\omega)}
\renewcommand*{\s}{(s)}
\newcommand*{\tr}{\mathrm{tr}}
\newcommand*{\AM}{\mathrm{AM}}
\newcommand*{\GM}{\mathrm{GM}}
\newcommand{\underbracedmatrix}[2]{%
	\left(
	\smash[b]{\underbrace{
			\begin{matrix}#1\end{matrix}
		}_{#2}}
	\right)
	\vphantom{\underbrace{\begin{matrix}#1\end{matrix}}_{#2}}
}
\newcommand{\minimatrix}[1]{\mbox{\tiny $\setlength{\arraycolsep}{2pt}\begin{pmatrix} #1 \end{pmatrix}$}}


\title{A Convex Approach to Inclusion Identification\\for Robust Control}
\title{A Convex Approach to Norm-Bounded Differential Inclusion Identification}
\title{An LMI Approach to Identifying Norm-Bounded Differential Inclusions}
\title{A Convex Approach to System Identification for Robust Control}
\title{The Convex ``Quadric Inclusion Program'':\\System Identification for Robust Control}
\title{Identifying $\hinf$-Models: An LMI Approach}
\title{Quadric Inclusion Programs: an LMI Approach to\\$\hinf$-Model Identification}

\author{Gray C. Thomas and Luis Sentis
\thanks{This work was supported by NASA Space Technology Research Fellowship grant NNX15AQ33H, ``Controlling Robots with a Spring in Their Step,'' for which Gray is the Fellow and Luis is the Advising Professor. Authors are with---respectively---the Departments of Mechanical and Aerospace Engineering,
        University of Texas at Austin, Austin, TX 78712, USA. Send correspondence to 
        {\tt\small gray.c.thomas@utexas.edu}}%
}
\markboth{Pre-Print Manuscript}%
{Shell \MakeLowercase{\textit{et al.}}: Bare Demo of IEEEtran.cls for IEEE Journals}

\maketitle


\IEEEpeerreviewmaketitle
\begin{abstract}
Practical application of $\hinf$ robust control relies on system identification of a valid model-set, described by a linear system in feedback with a stable norm-bounded uncertainty. This model-set should explain all possible (or at least all previously measured) behavior for the controlled plant. Such models can be viewed as norm-bounded inclusions in the frequency domain. This note introduces the ``Quadric Inclusion Program'' that can identify inclusions from input--output data as a convex problem. We prove several key properties of this algorithm and give a geometric interpretation for its behavior. While we stress that the inclusion fitting is outlier-sensitive by design, we offer a method to mitigate the effect of measurement noise. We apply this method to robustly approximate simulated frequency domain data using orthonormal basis functions. The result compares favorably with a least squares approach that satisfies the same data inclusion requirements. 
\end{abstract}






\section{Introduction}\label{intro}

\IEEEPARstart{S}{ystem} identification of a valid $\hinf$ plant model marks the first obstacle to applying the robust $\hinf$ control theory of e.g. \cite{ZhouDoyleGlover1996Book,ZhouDoyle1998Book,DullerudPaganini2013Book}. If this model is not believable then $\hinf$ synthesis provides guesses rather than guarantees---with the parameters of the uncertainty acting as tunable knobs. In many cases this is an acceptable strategy, but in some cases uncertainty demands more accurate measurement. Guesswork is done conservatively, and conservatism in the uncertainty model can degrade performance. When uncertainty is the performance-limiting factor, we expect this uncertainty model to represent some sort of physical limit to the plant. 

Most identification stems from the celebrated prediction error method \cite{Ljung1999Book} which produces high quality linear models complete with a measure of model certainty in the form of a model parameter covariance matrix. This parameter covariance, its implication for robust control, its improved value when using instrumental variables (or orthonormal basis function parameterizations), and the influence of weighting functions and closed loop identification controllers on it have all been extensively studied \cite{ZangBitmeadGevers1995Automatica,HakvoortVDHof1997TAC,ForssellLjung1999Automatica,AlbertosSala2002Book,BomboisGeversScorlettiAnderson2001Automatica}. This confidence measure is often taken out of context, however, as it represents only the distribution of models which would result from the same identification process if the data were regenerated. Prediction error uncertainty is not capable of representing the influence of a nonlinearity \cite{Toffner1996Book}. Moreover, with additional data the model parameter covariance will decrease even if the error variance is constant---a sought after property of consistency---but a property which clearly indicates that the parameter covariance is not a measure of any physical property. 
%

A paradigm known as stochastic embedding \cite{GoodwinSalgado1989ACSP,Toffner1996Book}, has been proposed to work around this---adding an additional source of uncertainty to the computation of parameter covariance. By supposing that the model parameters are sampled from a distribution with pre-defined covariance, the stochastic embedding approach estimates the means of these parameter distributions rather than the parameters themselves---and returns a much more conservative covariance estimate. This covariance does not approach zero with more samples---instead it approaches the a priori covariance.

The primary alternative to prediction error identification is broad-spectrum frequency-domain estimation \cite{PintelonSchoukens2012book}. This approach uses a ratio of the Fast Fourier Transform (FFT) spectra of the input to the output. To eliminate noise, the FFT data must be averaged in the frequency domain, often weighted by the magnitude of the input (or occasionally by the magnitude of the output)---making a ratio of cross spectrum to power spectrum. An uncertainty bound can be obtained by repeatedly generating estimates of the transfer function and then drawing a bound around them numerically \cite{ZhouDoyle1998Book}, but our method takes this further and simultaneously optimizes the model and the shape matrices. $\hinf$--oriented identification based on corrupted point-samples of the frequency response has been analyzed before in a single input, single output setting \cite{HelmickiJacobsonNett1990ACC}, \cite{HelmickiJacobsonNett1991TAC}, but the model is discrete time SISO, there are no shape matrices, and the approach assumes a unique true model (i.e, not a true model-set). 

The popular domain of model validation (through lack of invalidation) tests a priori model-sets on time domain data \cite{PoollaKhargonekarTikkuKrauseNagpal1994TAC}. This approach uses the Kalman-Yakubovitch-Popov Lemma to relate frequency domain bounds to time domain bounds on ``uncertainty'' signals, and tests for the satisfiability of those bounds using convex optimization (linear programming) within a finite horizon. An elegant approach to be sure, but not one which identifies model-sets. Adding flexibility in the model-set would make the problem non-convex.

In this technical note we introduce the Quadric Inclusion Program (QIP\footnote{Quadric Inclusion Programming is also abbreviated QIP.}), a convex program that identifies norm-bounded linear inclusions (a type of model-set with good scaling properties) from pairs of input and output vectors (Section \ref{sec:qip}). The next three sections develop geometric intuition for this program, demonstrate its consistent estimation, and introduce a method to ignore noise-dominated data points. Section \ref{sec:FDID2} applies the QIP methodology to identify a minimally uncertain model which explains the frequency domain measurements of a simulated mechanical system. This system has realistic imperfections, including input and sensor noise, inconsistent friction elements, and a time delay.


\section{The Quadric Inclusion Program}\label{sec:qip}

In this section we consider the linear norm-bounded inclusion $y\in\{(A+B\Delta C)x \ : \Delta^*\Delta\preceq I\}$ with $y\in\mathbb C^{n_y}$ and $x\in \mathbb C^{n_x}$ vectors, $A$ a real matrix of appropriate dimension, $B$ and $C$ real invertible square matrices, and complex valued $\Delta$ unknown but norm bounded\footnote{$\Delta^*\Delta\preceq I\iff \|\Delta\|\leq 1$.}: $\Delta^*\Delta\preceq I$. If the linear equation $y=Ax$ is geometrically analogous to a line, then this inclusion is analogous to a cone. In this section we present a convex program which finds the real-valued parameters of the inclusion, the elements of the $A$, $B$, and $C$ matrices, based on a series of measurements of $x$ and $y$. 

However, we must apply a lossless convexification, as the problem is not naturally amenable to convex optimization tools. And for this convexification we need to invoke an alternative form for this inclusion:

\begin{proposition}[Quadratic Form]\label{prop:quadric}
A pair of input and output vectors ($x$, $y$) satisfies the inclusion
\begin{align}
y\in \{ (A+B\Delta C)x\ :\ \|\Delta\|\leq 1\},\label{eq:base_inclusion}
\end{align}
with full rank $B$
if and only if it satisfies the following quadratic form inequality:
\begin{align}
\begin{pmatrix}
y\\x
\end{pmatrix}^*
\underbrace{
\begin{pmatrix}
-B^{-T}B^{-1} & B^{-T}B^{-1}A\\
A^TB^{-T}B^{-1} & C^TC-A^TB^{-T}B^{-1}A
\end{pmatrix}}_{\displaystyle{Q}}
\begin{pmatrix}
y\\x
\end{pmatrix}
\geq 0. \label{eq:quadric_form}
\end{align}
\end{proposition}

\begin{proof}
First, consider $x$ and $y$ which satisfy \eqref{eq:base_inclusion}:
\begin{align}
\Delta Cx &=B^{-1}(y-Ax),\\
\|Cx\|^2  &\geq \|\Delta C x\|^2 = \|B^{-1}(y-Ax)\|^2,\\
0 &\leq x^*C^TC x-  (y-Ax)^*B^{-T}B^{-1}(y-Ax) .
\end{align}
Which is equivalent to \eqref{eq:quadric_form}. Conversely, for $x$ and $y$ which satisfy \eqref{eq:quadric_form} (assuming $Cx\neq0$, noting that the trivial case where $Cx=0$ results in \eqref{eq:quadric_form}$\implies y=Ax\implies$\eqref{eq:base_inclusion}) we can choose $\Delta = B^{-1}(y-Ax)x^*C^T/(x^*C^TCx)$ to satisfy both $y = Ax + B \Delta C x$ (trivially) and $\Delta^*\Delta\preceq I$:
\begin{align}
\Delta^*\Delta &= \frac{C x (y-Ax)^* B^{-T}B^{-1}(y-Ax)x^* C^T}{(x^*C^TCx)^2}\\
&= \gamma \frac{C x x^* C^T}{(x^*C^TCx)}, \label{eq:CxxC}
\end{align}
with $\gamma = (y-Ax)^* B^{-T}B^{-1}(y-Ax)/(x^*C^TCx) \leq 1$ as $x^*C^TCx\geq (y-Ax)^* B^{-T}B^{-1}(y-Ax)$. The matrix in \eqref{eq:CxxC} is rank 1, and positive semi-definite, with $\gamma\leq1$ as the only non-zero eigenvalue. This ensures that  $\Delta^*\Delta\preceq I$.
\end{proof}

In this form, the inclusion has become a linear inequality constraint on the elements of $Q$. But not all symmetric matrices will have the appropriate structure to be interpreted as $Q$ for the purpose of backing out the inclusion matrices. Fortunately, we can re-parameterize around this issue by decomposing the matrix $Q$.

\begin{definition}[SS-DD]
The matrix $Q$ can be expressed as the difference of two positive semi-definite matrices. Using linear matrix inequality constraints on four new real-valued matrix variables, we can construct a similar structure to $Q$, which we call the {Split Semi-Definite Decomposition (SS-DD)}. The following three equations constrain the new SS-DD variables $X_B$, $X_{A}$, $X_{AA}$ and $X_C$:
\begin{align}
	0&\preceq Q^\prime \triangleq \begin{pmatrix}
	X_B& -X_{A}\\
	-X_{A}^T& X_{AA}
	\end{pmatrix},	\label{eq:SSDD:Qneg}\\
	0&\preceq X_C,	\label{eq:SSDD:Qpos}\\
	Q& = \begin{pmatrix}
	-X_B& X_{A}\\
	X_{A}^T& X_C-X_{AA}
	\end{pmatrix}. \label{eq:SSDD:Qsplit}
\end{align}
However, while all quadratic form $Q$ can be written using the four matrices of the SS-DD form, the converse is not true in general; the $Q$ of the SS-DD will only match the $Q$ structure in \eqref{eq:quadric_form} in a special case:
\begin{equation}
X_{AA}=X_{A}^TX_B^{-1}X_{A}, \label{eq:SSDD:special}
\end{equation}
that is, if we can write the RHS of \eqref{eq:SSDD:Qneg} as 
$
\begin{pmatrix}X_B & -X_{A}\end{pmatrix}^T$ $
\begin{pmatrix}X_B^{-1}\end{pmatrix}$ $
\begin{pmatrix}X_B & -X_{A}\end{pmatrix},
$
then the rank of the RHS of \eqref{eq:SSDD:Qneg} is the rank of $X_B$ and the SS-DD structure matches that of \eqref{eq:quadric_form}:
\begin{align}
\left\{\begin{matrix}
X_B &=& B^{-T}B^{-1},\\
X_{A}&=&B^{-T}B^{-1}A,\\
X_{AA}&=&A^TB^{-T}B^{-1}A,\\
X_C&=&C^TC.\end{matrix} \right.  \label{eq:SSDD:assignment}
\end{align}
\end{definition}

The SS-DD is a lossless convexification of the search space, and we will introduce a convex program which is formatted to follow the rules of disciplined convex programming  \cite{BoydVandenberghe2004Book} to find the SS-DD. However, that program must also result in satisfaction of \eqref{eq:SSDD:special} for all optimal solutions if the variables are to be interpreted as an inclusion. This is guaranteed for our QIP, but not for other cost functions or additional constraints involving the SS-DD. Note that the following problem uses the cost function elaborated in Sec.~\ref{sec:width} and is proven to be consistent in Sec.~\ref{sec:consistency}.
\begin{problem}[Degenerate Quadric\footnote{Named for the geometric shape explained in the next section.} Inclusion Program]\label{prob:degenerateQIP}
	The inclusion that minimizes a width-like cost\footnote{$\GM\{\sigma(B)\} \|C\|_{\mathrm{Frobenius}}$, where $\GM$ denotes the geometric mean; $\sigma(B)$, the spectrum, or set of singular values. The trace constraint on $X_C$ resolves a scale ambiguity in the SS-DD and allows the problem to be convex.} while including a list of data point pairs $\xi_i = (y_i^*,\ x_i^*)^*$, $i=1,\dotsc,N$ can be found by the following convex optimization program
	\begin{align}
	{\text{\rm maximize}}\quad &
	\log (\det (X_B))
	\nonumber\\
	\text{\rm over}\quad&Q,\ X_B,\ X_{A},\ X_{AA},\ X_{C} \nonumber\\
	\text{\rm subject to} \quad&\text{SS-DD equations \eqref{eq:SSDD:Qneg}-\eqref{eq:SSDD:Qsplit}} \\
	&  1 = \tr[X_C] \nonumber \\
	&  0 \leq \xi_{i}^* Q \xi_{i} \quad \forall\quad i\in 1\dots N \nonumber 
	\end{align}
\end{problem}

\begin{proposition} \label{prop:specialness}
	The SS-DD satisfies \eqref{eq:SSDD:special} for all solutions to Prob.~\ref{prob:degenerateQIP} with finite cost.
\end{proposition}
\begin{proof}
	In the general case where \eqref{eq:SSDD:special} does not hold, we can define an equation error matrix
	\begin{equation}
	\tilde X_{AA} \triangleq X_{AA}-X_{A}^TX_B^{-1}X_{A} \succeq 0,
	\end{equation}
	which is p.s.d. since it is a Schur complement of $Q^\prime$, the matrix in \eqref{eq:SSDD:Qneg}. Suppose $\tilde X_{AA}$ has a non-zero (real) eigenvalue $\lambda>0$ and corresponding (real-valued) eigenvector $v$. Now consider another potential solution identical in all ways but one to the previous solution: $X_{AA}^\prime = X_{AA}-v\lambda v^T$. Since the rest of the solution is unchanged, the new solution trivially satisfies all constraints which do not involve $X_{AA}$, leaving \eqref{eq:SSDD:Qneg}---satisfied by Schur complement since $\tilde X_{AA}-v\lambda v^T\succeq0$, and the inclusion inequalities, which the new solution relaxes to (recall that $x_i$ is the lower part of $\xi_i$):
	\begin{equation}
	0\leq \xi_i^*Q\xi_i+x_i^* v \lambda v^Tx_i. \label{eq:special_prop_check}
	\end{equation}
	Relaxing constraints increases the objective function for our maximization problem. Assuming that the objective was previously bounded by the inclusion inequalities in all the various ways it could increase, at least one of them must have been relaxed by the change (albeit potentially through re-arrangement of the $X_C$ matrix), demonstrating that a higher objective solution must exist---a contradiction of the premise that optimal solutions can have non-zero equation error for equation \eqref{eq:SSDD:special}.
\end{proof}

\begin{remark}[Trivial Solution]
	A trivial solution to the constraints always exists, with $Q^\prime=0$, $X_C=\lambda I\ :\ \tr[X_C]=1$. This is the worst possible solution, since it represents an inclusion with infinite uncertainty magnitude, and it has infinitely negative optimization function value.
\end{remark}

\begin{remark}[Linear Solutions]
	If there are insufficient data points, or if the data points share a perfect linear relationship, $X_B$ will have an unbounded eigenvalue. In this situation, the optimization function value will be infinite and the solution\footnote{if one is returned at all---data sets like this typically cause numerical solvers to fail.} will not be unique.
\end{remark}


\section{A Geometric Notion of Cone Width for Degenerate Quadric ``Cones''} \label{sec:width}

The set of all points $(x,y)$ satisfying the quadratic inequality \eqref{eq:quadric_form} has a geometric interpretation---a shape which is technically a degenerate quadric. Quadrics in 2D real space are familiar to many: non-degenerate varieties include ellipses, hyperbola, and circles; but cones are degenerate. Degenerate quadrics are a general class of hyperdimensional shapes: they are described by a symmetric matrix quadratic form inequality relative to zero. Non-degenerate quadrics are similar, but with this inequality relative to some non-zero constant. With some positive and some negative eigenvalues (separated in the SS-DD), our degenerate quadrics have a useful analog in the simple 2D cone. 

\begin{definition}[Cone width for a toy cone]
	Consider the following 2D cone in real scalar $x$, $y$ space: $|y-Ax|\leq r(x)=w|x|$. We call $r(x)$ the ``radius'' of the cone opening as a function of $x$, and $w$ the ``width'' of this cone. We can equivalently define this width as $\sqrt{\mathrm E(r(x)^2)}$ if $x\sim \mathcal N(0,1)$, where $\mathrm E$ is expected value, and $\mathcal N(0,1)$ is the normal distribution with mean $0$ and variance 1.
\end{definition}
\begin{proof}
	$
	\sqrt{\mathrm E\left(r(x)^2\right)}=\sqrt{w^2\mathrm E(|x|^2)}=w.
	$
\end{proof}


\begin{definition}[Characteristic radius of a cross section]
	Consider a special case SS-DD satisfying \eqref{eq:SSDD:Qsplit}--\eqref{eq:SSDD:assignment} such that it is equivalent to an inclusion. If we specify a particular input $x$, the space of included $y$ can be interpreted as a geometric shape: a hyper-ellipsoid. The characteristic radius $R(x)$ of this cross section is defined as the radius of the hyper-ball that has equal hyper-volume to this hyper-ellipsoid. The hyper ellipsoid can be described as
	\begin{equation}
	(y-Ax)^*X_B (y-Ax) \leq G^2(x),\quad G(x)\triangleq \sqrt{x^* X_C x};
	\end{equation}
	and the characteristic radius,
	$
	R(x) = G(x)$ $(\det(X_B^{-1}))^{\frac 1 {2 n_y}}, 
	$
	 where $y\in \mathbb C^{n_y}$.
\end{definition}

\begin{definition}[Generalized Cone Width]
	The {generalized cone width} is the square root of the expected value of the squared characteristic radius given inputs drawn from the standard complex multivariate normal distribution. That is, \begin{equation}W \triangleq \sqrt{\mathrm E\left(R^2(x)\right)}\quad \mathrm{with} \quad \mathrm E\left(xx^*\right)=I.\end{equation}
\end{definition}

\begin{theorem}[Generalized Cone Width]\label{thm:W}
	The generalized cone width of a special case SS-DD form satisfying \eqref{eq:SSDD:Qsplit}--\eqref{eq:SSDD:assignment} \begin{equation}W=(\det{\left(X_B^{-1}\right)})^{\frac{1}{2n_y}}\sqrt{\tr \left( X_C\right)}.\label{eq:gcw}\end{equation}
\end{theorem}
\begin{proof}
	The width,
	\begin{align}
	W = \sqrt{\mathrm E\left(R^2(x)\right)}&=(\det(X_B^{-1}))^{\frac 1 {2n_y}}\sqrt{\mathrm E(G^2(x))};\\
	\mathrm E\left(G^2(x)\right)&=\mathrm E\left(x^* X_C x\right) \\&= \tr[ X_C \mathrm E\left(xx^*\right)]=\tr[ X_C].
	\end{align}
	Substitution yields \eqref{eq:gcw}.
\end{proof}

We use the following corollary to avoid having both $X_C$ and $X_B$ in the cost function of the QIP.
\begin{corollary} \label{corr:logwidth}
	If $\tr (X_C)=1$,
	\begin{equation}
	\log(W)= -\frac1{2n_y}\log(\det(X_B))
	\end{equation}
\end{corollary}
\begin{proof}
	Using \eqref{eq:gcw}, $\log(W)= \frac{1}{2n_y}\log\left(\det\left(X_B^{-1}\right)\sqrt{1}\right) =\frac{-1}{2n_y}\log\left(\det\left(X_B\right)\right)$.
\end{proof}
The maximization objective $\log (\det (X_B))$ in Prob.~\ref{prob:degenerateQIP} is a negative multiple of this expression for generalized cone width (and therefore minimizes it). As will be shown in Corollary \ref{cor:scale}, this trace constraint specifies a free scale parameter in the SS-DD.

\section{Proof of Consistent Estimation} \label{sec:consistency}
In this section we use properties of the generalized cone width to prove that, when data is generated from a norm-bounded linear inclusion, the estimates from the degenerate QIP converge, in a certain sense, to equivalence with the true inclusion. This is essentially a property of the choice of cost function, and could be otherwise stated ``every norm-bounded inclusion is optimal for the data it produces''.\footnote{Without this property one could imagine a cost function that always opts for sphere-like models, scaled identity matrices for $B$ and $C$: any information about the shape of the uncertainty in the data would be ignored.} We offer this proof only for invertible $C$ matrices, though we conjecture the proof could be made to work using the pseudo-inverse\footnote{Our simulation example seems to be consistent in the absence of an invertible $C$ matrix.}.

Norm-bounded linear inclusions are functionally equivalent (include the same points) up to an orthogonal pre-multiplication of $C$, an orthogonal post-multiplication of $B$, and reciprocal scaling of $C$ and $B$, and the cone width does not change due to any such alteration.
\begin{corollary}
	Generalized cone width is the product of the geometric mean of the singular values of $B$, and the 2-norm of the singular values (the Frobenius norm) of $C$
\end{corollary}
\begin{proof}
	Converting the model-set from standard form \eqref{eq:base_inclusion} to the special case of the SS-DD form \eqref{eq:SSDD:Qneg}--\eqref{eq:SSDD:special} and applying Thm.~\ref{thm:W},
	\begin{align}
	W &= (\det{BB^T})^{\frac1{2n_y}}\sqrt{\tr[C^TC]},\\
	&= \prod\limits_{\lambda \in \sigma(B)}\left(\lambda^{\frac{1}{n_y}}\right) \sqrt{\sum\limits_{\gamma \in \sigma(C)}\gamma^2}.
	\end{align}
	Where the spectrum of a matrix $\sigma(\cdot)$ is the set of the singular values of that matrix (with repetition). 
\end{proof}
Multiplication by a orthogonal matrix cannot change the singular values of a matrix, and cannot influence the cone width.

\begin{corollary} \label{cor:scale}
	The generalized cone width is invariant to scaling $C$ and $B$ by reciprocal values.
\end{corollary}
\begin{proof}
	$
	W^\prime = \GM\{\sigma(\alpha^{-1}B)\} \|\alpha C\|_{\mathrm{Frobenius}} =  \frac{\alpha}{\alpha} W.
	$
\end{proof}

The cone width also satisfies an intuitive notion that a cone can only contain another cone if it is wider. In terms of inclusions, this geometric containment becomes a concept of implication: if one inclusion is implied by a second, this is equivalent to saying that the second inclusion is geometrically contained within the first.


\begin{lemma}
	An inclusion  $y\in\{(A_o+B_o\Delta C_o)x \ : \Delta^*\Delta\preceq I\}$ (subscript $o$ for outer) contains another inclusion $y\in\{(A_i+B_i\Delta C_i)x \ : \Delta^*\Delta\preceq I\}$ (subscript $i$ for inner) if and only if
	\begin{equation}
	\|\tilde A +\tilde B \Delta \tilde C\|\leq 1\ \ \forall\ \Delta \ |\ \|\Delta\|\leq 1, \label{eq:inclusion_norm}
	\end{equation}
	with $\tilde B \triangleq B_o^{-1}B_i$, $\tilde A \triangleq B_o^{-1}(A_i-A_o)C_o^{-1}$, and $\tilde C \triangleq C_iC_o^{-1}$.
\end{lemma}
\begin{proof}
	Geometric containment of inclusion shapes is logical implication of inclusion inequalities:
	\begin{gather}
	\left\{
	\begin{matrix}
	y = A_ox + B_o\Delta_o C_o x,\\
	\mathrm{with}\ \Delta_o^*\Delta_o \preceq I 
	\end{matrix}
	\right.\impliedby \left\{
	\begin{matrix}
	y = A_ix + B_i\Delta C_i x,\\
	\mathrm{with}\ \Delta^*\Delta \preceq I. 
	\end{matrix}
	\right.\label{eq:includes}
	\end{gather}
	The two equalities define a relationship between $\Delta$ and $\Delta_o$, which we can equivalently state:
	\begin{gather}
	\left\{
	\begin{matrix}
	\Delta_o^*\Delta_o \preceq I \impliedby \Delta^*\Delta \preceq I,  \\
	A_o + B_o\Delta_o C_o = A_i + B_i\Delta C_i.
	\end{matrix}
	\right. \label{eq:bracket_implication}
	\end{gather} 
	This is because \eqref{eq:includes} holds for all $x$.
	By algebra, $	\Delta_o =  \tilde A + \tilde B \Delta \tilde C$.
	Re-stating the implication in \eqref{eq:bracket_implication} as $\|\Delta_o\|\leq 1 \ \forall\ \Delta\ |\ \|\Delta\|\leq 1$ we get \eqref{eq:inclusion_norm}.
\end{proof}

This eventually leads to a necessary condition for inclusion based on a cone width inequality. An intermediate necessary condition uses the singular values of the $B$ and $C$ matrices.

\begin{lemma} 
	If $y\in\{(A_o+B_o\Delta C_o)x \ : \Delta^*\Delta\preceq I\}$ includes $y$ $\in\{(A_i+B_i\Delta C_i)x \ : \Delta^*\Delta\preceq I\}$ then $\sigma_{\max}(\tilde B)\sigma_{\max}(\tilde C)\leq 1$.\label{lem:inclusion}
\end{lemma}
\begin{proof}
	Assume the contrary ($\sigma_{\max}(\tilde B)\sigma_{\max}(\tilde C)> 1$) and construct the following:
	\begin{equation}
	\Delta = \tilde B^T \zeta_B \sign(\zeta_B^T\tilde A \zeta_C)\zeta_C^T\tilde C^T,
	\end{equation}
	where $\zeta_B$ and $\zeta_C$ are (real-valued) unit eigenvectors for $\tilde B \tilde B^T$ and $\tilde C^T\tilde C$ corresponding to their respective maximum eigenvalues: $\lambda_{\max}(\tilde B\tilde B^T) \zeta_B = \tilde B \tilde B^T \zeta_B$,  $\lambda_{\max}(\tilde C^T \tilde C)\zeta_C = \tilde C^T\tilde C \zeta_C$, $\zeta_B^T\zeta_B=\zeta_C^T\zeta_C=1$. This choice of $\Delta$ leads to:
	\begin{gather}
	1\geq \|\tilde A +\tilde B \Delta \tilde C\|\geq |\zeta_B^T (\tilde A +\tilde B \Delta \tilde C)\zeta_C| \nonumber\\
	= |\zeta_B^T (\tilde A +\tilde B \tilde B^T \zeta_B \sign(\zeta_B^T\tilde A \zeta_C)\zeta_C^T\tilde C^T  \tilde C)\zeta_C|
	\nonumber \\
	= |\zeta_B^T\tilde A\zeta_C| +\lambda_{\max}(\tilde B \tilde B^T) \lambda_{\max}(\tilde C^T \tilde C), \label{eq:bound_with_p} \\
	\geq (\sigma_{\max}(\tilde B)\sigma_{\max}(\tilde C))^2 > 1.
	\end{gather}
	A contradiction, as desired. 
\end{proof}

\begin{proposition}
	If the outer inclusion contains the inner inclusion (with generalized cone widths $W_o$ and $W_i$, respectively), then $W_o\geq W_i$.
\end{proposition}
\begin{proof}
	Assume the contrary ($W_i>W_o$), 
	\begin{gather}
	\sqrt[2n_y]{\det{B_iB_i^T}}\|C_i\|_{\mathrm F}
	>
	\sqrt[2n_y]{\det{B_oB_o^T}}\|C_o\|_{\mathrm F}, \\
	\frac{\sqrt[n_y]{\det{B_iB_i^T}}}{\sqrt[n_y]{\det{B_oB_o^T}}}\|C_i\|_{\mathrm F}^2
	>
	\|C_o\|_{\mathrm F}^2, \\
	\tr\left(\sqrt[n_y]{\det(B_o^{-1}B_iB_i^TB_o^{-T})}C_i^TC_i - C_o^TC_o\right)>0,\label{eq:trace_interesting_matrix}\\
	\tr\left(C_o^T\left(\lambda_{\max}(\tilde B\tilde B^T)\tilde C^T\tilde C - I \right)C_o\right)>0 .
	\end{gather}
	However by Lemma~\ref{lem:inclusion}, the argument of trace in the above inequality is negative semi definite, so it can not have a positive trace---a contradiction.
\end{proof}

The special case where the cone widths reach equality marks the residual set of a Lyapunov-like argument in the consistency proof. Conveniently, this residual set has only one element.

\begin{proposition} \label{prop:equal_W}
	If $W_o=W_i$ then our two inclusions are equivalent in the sense that $A_o=A_i$, and $\exists\ \lambda>0 \ : \lambda B_oB_o^T = B_iB_i^T, \ C_o^TC_o=\lambda C_i^TC_i$.
\end{proposition}
\begin{proof}
	When the two widths are equal, the derivation which produced  \eqref{eq:trace_interesting_matrix} yields:
	\begin{gather}
	\tr\left(C_o^T\left(\sqrt[n_y]{\det{(\tilde B\tilde B^T)}}\tilde C^T\tilde C - I \right)C_o\right)=0. \label{eq:trace_pos}
	\end{gather}
	Yet as before, singular values are limited by the inclusion constraint, and this guarantees \eqref{eq:bound_with_p} and a long series of matrix inequalities,
	\begin{gather}
	I\succeq |\zeta_B^T\tilde A\zeta_C|I +\lambda_{\max}(\tilde B \tilde B^T) \lambda_{\max}(\tilde C^T \tilde C)I \nonumber\\
	\succeq \lambda_{\max}(\tilde B\tilde B^T)\tilde C^T\tilde C\succeq \sqrt[n_y]{\det{(\tilde B\tilde B^T)}}\tilde C^T\tilde C = I, \label{eq:long_series}
	\end{gather}
	with this last equality due to the combination of \eqref{eq:trace_pos} and the last inequality above (which can be extended to
	\begin{gather}
	C_o^T \left(\sqrt[n_y]{\det{(\tilde B\tilde B^T)}}\tilde C^T\tilde C - I \right)C_o \preceq 0,
	\end{gather}
	another negative semi-definite matrix) ultimately forcing the inner matrix difference to be zero (as it is both negative semi-definite and has trace zero). 
	
	With both the first and last element identity, \eqref{eq:long_series} is actually a long chain of equalities. This gives
	\begin{gather}
	\lambda\triangleq \lambda_{max}(\tilde B \tilde B^T)=\sqrt[n_y]{\det{(\tilde B\tilde B^T)}},\\
	\quad |\zeta_B^T\tilde A\zeta_C|=0, \quad \lambda \tilde C^T \tilde C =I,
	\end{gather} 
	from which it follows that $C_o^TC_o=\lambda C_i^TC_i$.
	When the geometric mean of the eigenvalues is equal to the largest eigenvalue, all the eigenvalues must be equal; thus $\tilde B \tilde B^T = \lambda I$ (or equivalently $\lambda B_oB_o^T = B_iB_i^T$). Since both $\tilde B \tilde B^T$ and $\tilde C^T\tilde C$ have only one eigenvalue with high multiplicity, the eigenvectors $\zeta_B$ and $\zeta_C$ can be any unit vectors. This in turn guarantees  $\tilde A=0$, that is,  $A_o=A_i$, completing the conditions necessary for the two model-sets to be equivalent. 
\end{proof}


Using these preliminaries, we can prove the following notion of estimation consistency, noting that without some knowledge of how frequently the true inclusion generates extreme data---data on the very edge of the inclusion---it is impossible to claim any rate of convergence.

\begin{theorem}[Estimation Consistency]\label{thm:consistency}
	Consider an infinite list of input output data $\xi_i=[y_i^*,x_i^*]^*\ \forall\ i\in\mathbb N$ points from the inclusion $y_i\in\{(A_T+B_T\Delta C_T)x_i \ : \Delta^*\Delta\preceq I\}$ with generalized cone width $W_T$ (subscript $T$ for true) in the sense that any possible output will eventually be produced within a non-zero tolerance. Suppose that inclusion estimates $y\in\{(A_n+B_n\Delta C_n)x \ : \Delta^*\Delta\preceq I\}$ with generalized cone width $W_n$ are calculated via Prob.~\ref{prob:degenerateQIP} using the subset of data indexes $i=1,\dots,n$ as $n$ increases towards infinity. 
	
	Then $W_{n+1}\geq W_n\ \forall\ n\in \mathbb N$, and
	$W_n\leq W_T$.
	
	Most importantly, the identification procedure is consistent in the sense that if $\exists\ n^\prime$ such that the inclusion width stops changing, $W_n=W_{n^\prime}\ \forall\ n\geq n^\prime$, then the $n^{\prime\text{th}}$ result inclusion must be equivalent to the generating inclusion.
\end{theorem}

\begin{proof}
	The first claim follows from the nature of the maximization: more constraints can only reduce the objective, this objective is proportional to the negative log of the width, and log is monotonic. The second is a consequence of the true inclusion being a feasible solution to the optimization problem: the optimal solution has the maximal objective over all feasible solutions. As for the third, suppose to the contrary that the inclusions are distinct. The $n^{\prime\text{th}}$ result inclusion cannot contain the true inclusion because it has lesser or equal cone width and is not the same (by supposition). There must be points within the true inclusion and outside the $n^{\prime\text{th}}$ result inclusion. And these points, which will eventually occur for some $n>n^\prime$, will not satisfy the inclusion inequalities with the $n^{\prime\text{th}}$ result inclusion---contradicting the notion that the estimates could stop changing without reaching the true inclusion.
\end{proof}

Convergence to a non-trivial inclusion is an important distinguishing aspect of this style of identification. Inclusions which are built on the error estimates in a least-squares fit \cite{Ljung1999Book,AlbertosSala2002Book,HeubergerHofWahlberg2005book} notably lack this property---converging instead towards a unique model (one element inclusion) as the estimated parameter covariance vanishes with additional samples.

%

\section{A Non-Degenerate Quadric for Noisy Measurements} \label{sec:noise}

Introducing noise into the complex output vector $y$ requires some notation for complex-value statistical distributions. To benefit from the more familiar notation of real-value distributions, let us introduce a simple bijection $\phi:\mathbb C^n \mapsto \mathbb R^{2n}$ which stacks the real part above the imaginary part: $\phi(y)=\begin{pmatrix}(y^T+y^*)/2 & (y^T-y^*)/2j)\end{pmatrix}^T$, with $j$ the imaginary unit. Using this bijection (and the obvious inverse) we will describe distributions over complex numbers as distributions over real numbers.

We assume the physical system is repeatable (though potentially corrupted by noise) if a {condition vector} \cite{ThomasSentis2017ACC}---comprised of the factors that cause the real system to deviate from a unique linear model---is held constant.\footnote{These factors are, for example, the input signal amplitudes and operating points for sinusoidal tests on nonlinear systems, and exogenous signals such as pressure or Mach number for aircraft.}
Formally, let the real-valued representation of the measurement $\phi(\hat y)$, with hat notation indicating a measurement, be a deterministic function of a condition vector $c$ that includes $x$, corrupted by stochastic zero-mean (real) measurement noise $\eta$ (from a potentially $c$-dependent distribution): $\phi(\hat {y})=f(c)+\eta$. These assumptions allow us to average multiple tests and take advantage of the central limit theorem: as the number of averaged samples grows, the distribution of the average approaches a normal distribution and its covariance shrinks towards zero. With $N$ samples of $\hat {y}$, indexed $\hat{y}_n,\ n=1,\dotsc,N$, we can find both the (real vector representation's) sample mean, $\bar {\mathrm y} \triangleq \sum^N_{n=1} \phi(\hat {y}_n)/N$, and sample covariance. The covariance of this mean itself can then be estimated similarly to the sample covariance, $\Sigma_\eta \triangleq \sum_{n=1}^N (\phi(\hat {y}_n)-\bar {\mathrm y})(\phi(\hat {y}_n)-\bar {\mathrm y})^T/(N^2-N) $. Note that, due to this averaging, we can expect that the noisy measurement case approaches the noiseless case as the number of averaged samples increases.

Of course, practical limits on the number of samples force us to consider the intermediate case where noise is small and normal, but not entirely eliminated. That is, data $\bar {\mathrm y} = f(c)+\bar \eta$, with $\bar{\eta}\sim \mathcal N(0,\Sigma_\eta)$.  In this scenario, our degenerate quadric model is inflexible near zero input: a zero $x$ must produce a zero $y$, and a near-zero $x$ must produce a near-zero $y$ unless the cone is preposterously wide. But the noise-corrupted average $\bar {\mathrm y}$ can take on a non-zero value even for zero ${\mathrm y}$. To address this issue, we introduce the idea of fitting a non-degenerate quadric relaxation of the model.

Returning to complex valued vectors\footnote{Bearing in mind that our bar decoration, e.g. $\bar \xi$, still refers to an average, not a complex conjugate.}, we distinguish the true input output data pair $\xi\triangleq(y^*\ x^*)^*$ from its averaged measurement $\bar \xi\triangleq \begin{pmatrix}\phi^{-1}(\bar{\mathrm y})^* & x^*
\end{pmatrix}^*$. This allows us to represent the output-noise-corrupted version of \eqref{eq:quadric_form}: 
\begin{align}
\xi^* Q \xi&=\left(\bar \xi- \begin{pmatrix}\phi^{-1} (\bar \eta)\\0\end{pmatrix}\right)^*Q\left(\bar\xi - \begin{pmatrix}\phi^{-1} (\bar \eta)\\0\end{pmatrix}\right)\geq 0; \\
\bar \xi^* Q \bar \xi &\geq \left[\xi^* Q\begin{pmatrix}\phi^{-1} (\bar \eta)\\0\end{pmatrix}
+\begin{pmatrix}\phi^{-1} (\bar \eta)\\0\end{pmatrix}^* Q\xi \right]\nonumber\\ &+ \begin{pmatrix}\phi^{-1} (\bar \eta)\\0\end{pmatrix}^* Q \begin{pmatrix}\phi^{-1} (\bar \eta)\\0\end{pmatrix}.
\end{align}
That is, substituting noisy data for perfect data shifts the threshold for cone inclusion. We consider the two RHS terms separately to find an appropriately negative lower bound for this new threshold. 

The first term is real, linear in $\bar \eta$, and normal.
However this term vanishes near the origin, where the effects of noise are the most problematic. It is therefore reasonable to ignore this term. The risk of over-conservatism due to inaccurately compensated noise in high-input magnitude data is mitigated by the naturally higher signal-to-noise ratio for this data and the promise that additional averaging can improve the signal-to-noise ratio on such data.

The second term is non-positive, and represents all noise near the origin. It is lower-bounded in magnitude by a chi-square distributed value and an expression which is linear in $X_B$:
\begin{gather}
\begin{pmatrix}\phi^{-1} (\bar \eta)\\0\end{pmatrix}^* Q \begin{pmatrix}\phi^{-1} (\bar \eta)\\0\end{pmatrix}=\bar\eta^T  \underbracedmatrix{-X_B&0\\0&-X_B}{I_2\otimes X_B }\bar\eta,
\nonumber\\
\geq - \|\Sigma_{1/2}(I_2\otimes X_B)\Sigma_{1/2}\| \|\Sigma_{1/2}^{-1}\eta\|^2,
\nonumber\\
\geq - \tr\left[\Sigma_\eta (I_2\otimes X_B)\right] \nu,\quad \nu\sim \chi_{2n_y}^2,
\end{gather}
with notation $\otimes$ for the Kronecker product and $\Sigma_{1/2}\succeq0$ the matrix square root of $\Sigma_\eta$ (satisfying $\Sigma_{1/2}\Sigma_{1/2}=\Sigma_\eta$).

Choosing a constant threshold $\alpha$ based on the survival function of $\chi_{2n_y}^2$, the inclusion's quadratic form inequality threshold is shifted to provide an arbitrarily low chance of feasibility problems with low-magnitude inputs:
\begin{problem}[(Non-Degenerate) Quadric Inclusion Program]\label{prob:QIP}
	\begin{align}
	{\text{\rm maximize}}\quad &
	\log (\det (X_B))
	\nonumber\\
	\text{\rm over}\quad&Q,\ X_B,\ X_{A},\ X_{AA},\ X_{C} \nonumber\\
	\text{\rm subject to} \quad&\text{SS-DD equations \eqref{eq:SSDD:Qneg}--\eqref{eq:SSDD:Qsplit}} \label{eq:ndQIP}\\
	&  1 = \tr[X_C] \nonumber \\
	&  0 \leq \alpha \tr[\Sigma_\eta (I_2\otimes X_B)]+ \xi_{i}^* Q \xi_{i} \quad \forall\quad i\in 1\dots N \nonumber 
	\end{align}
\end{problem}

Note that, since this modification changes the constraints of the problem, it requires re-examining Prop.~\ref{prop:specialness}. Fortunately, \eqref{eq:special_prop_check} does not gain any terms which would invalidate the proof as a result of this noise modification, and we can therefore accept it safely.

As shown in Fig.~\ref{fig:cartoon}, the non-degenerate quadric can be visualized as a hyperbola-bounded region, which asymptotically approaches the original, degenerate, quadric at large amplitudes. As noise-magnitude is reduced through averaging repeated measurements, the deviation between this approximate ``hyperboloid'' relaxation and a potentially more accurate Minkowski-sum style noise-relaxation becomes less significant. The QIP is not consistent in the sense that the degenerate QIP is, but the non-degenerate QIP approaches the degenerate one as noise is averaged away.\footnote{Consistency in the presence of deterministic noise may be possible as an extension of the degenerate QIP, however.}

\begin{figure}
	\centering
	\includegraphics[width=.82\columnwidth]{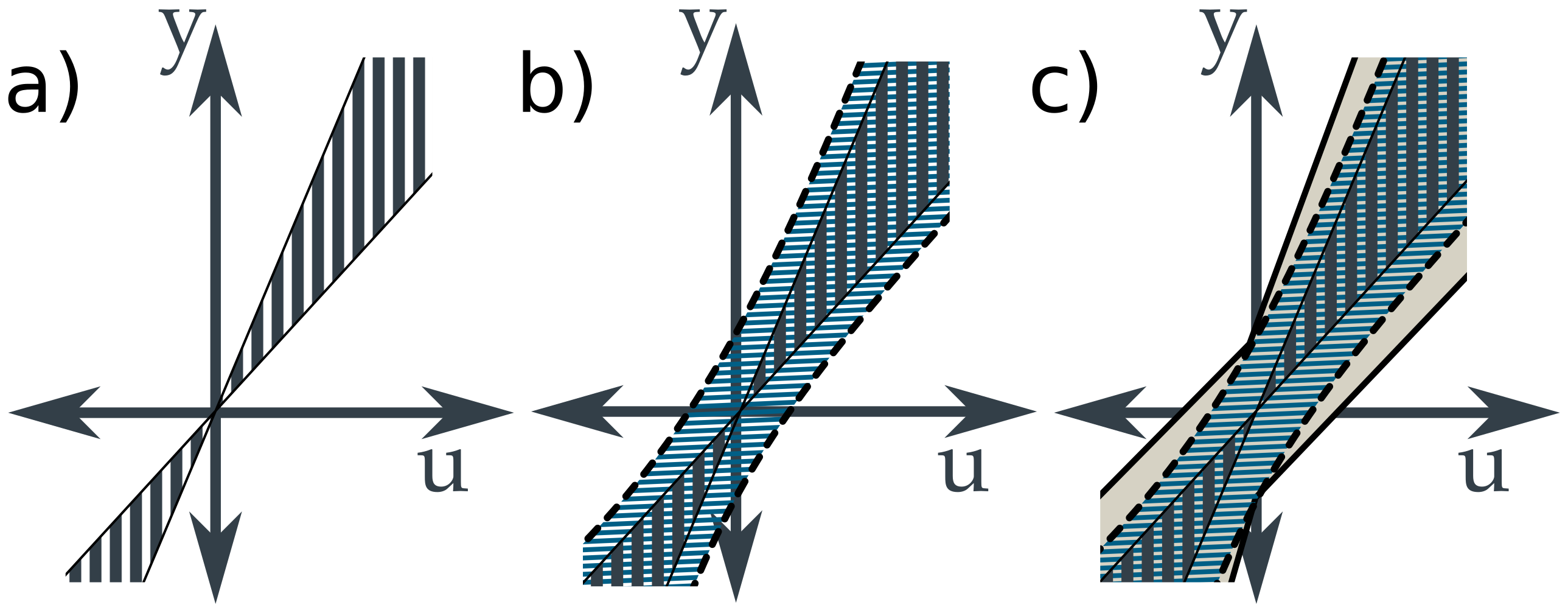}
	\begin{minipage}{3.5in}
%
%
%
%
	\end{minipage}
	\caption{Illustration of a hyperboloid cone. (a) a (degenerate quadric) cone in 2D. (b) a  hyperboloid (non-degenerate quadric) cone is overlaid over the original cone. (c) a properly (Minkowski sum) inflated cone is overlaid on the previous two cones.}\label{fig:cartoon}
\end{figure}

\section{Visualizing QIP in 2D}\label{sec:qip_disc}

Fig.~\ref{fig:anscombe} shows the QIP fit for a standard statistics data set \cite{Anscombe1973TAS}. Unlike Least Squares, the QIP is outlier-sensitive, as it bounds worst-case behavior. Only two data points in each plot lie on the non-degenerate quadric boundary, and these data determine the final result. Hence the caution with which we urge averaging repeatable tests to ensure that each point in the QIP has as little noise as possible. 
\begin{figure}
	\centering
	\begin{minipage}{3.5in}
		\centering\resizebox{3in}{!}{
			\def\svgwidth{3 in}
		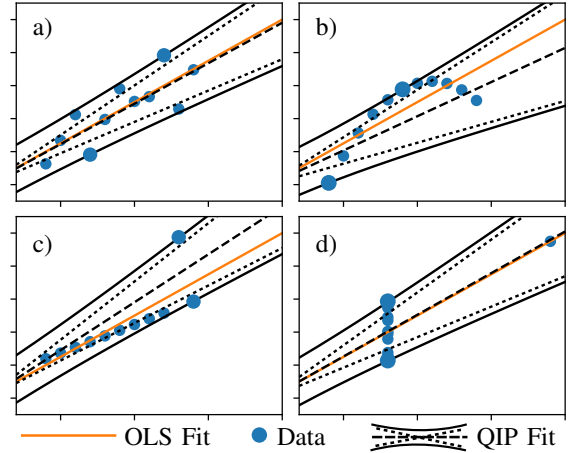}
	\end{minipage}
	\caption{Anscombe's Quartet \cite{Anscombe1973TAS}, a set of four lists of x-y points that have all the same statistics, yet have very different underlying data. QIP fits (with $\alpha \Sigma_\eta=2$) differ in each plot. The data is shifted from Anscombe's original (-3 in $y$) so that the least squares fit lines intersect the origin. QIP fit displayed with nominal model (dashed line), the degenerate quadric asymptote (cone-like dotted lines), and non-degenerate quadric bounds---which must include the data---(solid lines in the shape of a hyperbola). The two data points in each plot which are highlighted (enlarged) are the critical points which determine the width of the QIP fit's cone. These points correspond to the active set of inequality constraints in Prob.~\ref{prob:QIP}.}\label{fig:anscombe}
\end{figure}

	\begin{figure*}%
		\centering%
		\scalebox{1.0}{%
			\def\svgwidth{1.0\textwidth}%
			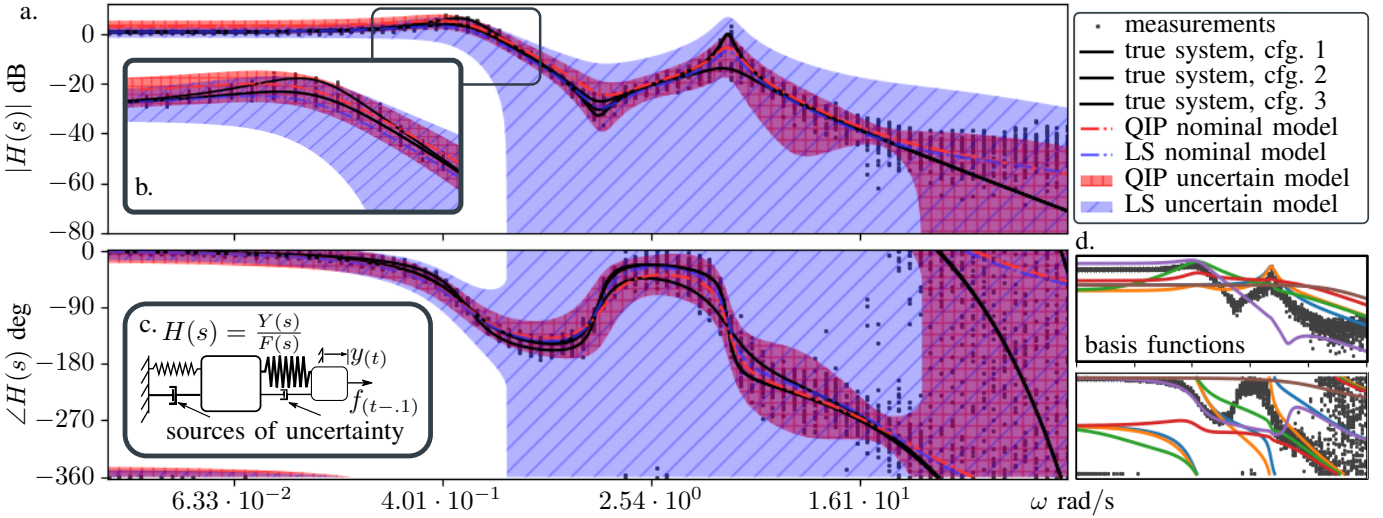%
		}%
		\vspace{-9pt}
		\caption{Inclusion based identification using an orthonormal basis.
		}\label{fig:experiment}%
	\end{figure*}

Reading Fig.~\ref{fig:anscombe} left to right, top to bottom, we see four scenarios that have the same ordinary least squares (OLS) best fit line (and error statistics), but different QIP fits. Note that the data set does not come with uncertain points, so a non-zero relaxation ($\alpha\Sigma_\eta=2$) has been applied to allow visualization of the hyperboloid behavior. In plot (a), the data are essentially a noisy line and both OLS and QIP find it. But, since QIP is inflated using a degenerate quadric instead of a Minkowski sum, it still describes these essentially linear data as a cone. In plot (b) there is a non-linear pattern to the data, and OLS approximates them in a somewhat meaningless way, while QIP fits a bound to the observed data. With a lower $\alpha\Sigma_\eta$, QIP would produce a very wide cone due to the leftmost data point. In (c), the data essentially lie in a cone, which QIP finds---but only approximately, given the nonzero $\alpha\Sigma_\eta$. OLS, fitting a line, leans towards the side which has more data points---a bias that QIP ignores. The final plot, designed to show how OLS can draw slope inferences from only the single outlier, a problem QIP doesn't have, because it has a fixed origin---its slope is determined by the extreme-most elements, which happen to lie in the big central stack.

This sensitivity to outliers suggests the possibility of more efficient testing if QIP fitting were paired with a machine learning system to select the next test condition (perhaps estimating the parameter--realization maps of \cite{Toth2010Book}), rather than collecting all data beforehand.

\section{Example: Transfer Function Approximation} \label{sec:FDID2}
In this section we demonstrate how QIP can be applied to solve a transfer function approximation problem, using a basis of transfer functions. The results are visualized in Fig~\ref{fig:experiment}.

The mechanical schematic shown in Fig.~\ref{fig:experiment}.c, with force input $f$ and position output $y$, has two second-order pole pairs representing the natural modes of the system. This system has a .1 second time-delay. In addition, the system has friction-related uncertainty, which we simulate as a condition vector dependent damping ratio for the two second-order pole pairs. Three experiments are simulated, each with different parameters for the damping ratios of these pole pairs, and these three true system configurations are plotted in Fig.~\ref{fig:experiment}.a. Together, these imperfections simulate the physical limitations which our norm-bounded inclusion model should represent.

Each experiment measures---for 100 angular frequencies evenly log-spaced between .01 and 100 rad/s---an input vector, a mean output vector, and the covariance matrix for the mean output vector (as per \cite{ThomasSentis2017ACC} and Sec.~\ref{sec:noise}).
The three experiments differ in their condition vector, which represents whichever under-modeled factors are responsible for changing the damping ratios of the pole pairs. For realism, we introduce white measurement and input noise averaged over the observational window of the ``Single Period Phasor Transform'' as defined in \cite{ThomasSentis2017ACC}. The noisy frequency response measurements for the three configurations are scatter plotted in Fig.~\ref{fig:experiment}.a.   

These same data are replicated in Fig.~\ref{fig:experiment}.d. This subplot compares the raw data with the six basis functions we use to parameterize the uncertain model-set approximation of the system. Good basis function selection revolves around approximation of the system poles with the basis functions \cite{HeubergerHofWahlberg2005book}. The poles of these basis functions were chosen to loosely approximate the poles of the system, plus one low pass filter to approximate the time delay, and one extra high frequency pole to approximate a bias. Perfect pole matching is impossible, as the system has both a time delay and changing friction parameters. We denote these six functions as the elements of a one-input, six-output (sixth order) transfer matrix $\mathbf C(sI-\mathbf A)^{-1}\mathbf B$. These transfer functions are designed to be orthonormal in $\mathcal H_2$ by appropriate selection of $\mathbf C$.\footnote{We find controllability Gramian $\mathbf W_c$ to satisfy $\mathbf W_c=\mathbf W_c^T\succ0,$ $ \mathbf W_c \mathbf A+\mathbf A^T \mathbf W_c+\mathbf B \mathbf B^T=0,$ and choose $\mathbf C$ such that $\mathbf C \mathbf W_c \mathbf C^T = I.$}

Our QIP method can be directly applied to this problem by multiplying the uncertain model of equation \eqref{eq:base_inclusion} by the basis function generator as follows
\begin{gather}
\underbracedmatrix{Y(s)}{y}=\left(A+B\Delta(s)C\right)\underbracedmatrix{\mathbf C (sI-\mathbf A)^{-1}\mathbf B F(s)}{x}, \label{eq:fixedpoleID}
\end{gather}
\par\vspace{-7pt}\hspace{-1em}%
where $F(s),\ Y(s)\ \in \mathbb C^1$ are the Laplace transforms of input-force and output-position respectively. And we solve this problem (in about 4.6 seconds) using CVXPY \cite{cvxpy}.

The QIP fit model is shown in Fig.~\ref{fig:experiment}.a alongside a model fit using weighted least squares (LS). The LS uncertain model, $Y_{LS}(s) = (A_{LS} + B_{LS}\Delta(s) C_{LS})\mathbf C(sI-\mathbf A)^{-1}\mathbf B F(s)$ uses the least squares parameter estimate as the nominal model $A_{LS}$, chooses $B_{LS}\triangleq 1$ and finds $C_{LS}^TC_{LS}\triangleq \Sigma_A$ according to the least squares parameter co-variance $\Sigma_A$. This uncertainty was then scaled to satisfy all the same inequalities required of the non-degenerate quadric inclusion program in Prob.~\ref{prob:QIP}---to ensure a fair comparison. The minimum such scaling is 10,860. The least squares fit was weighted using the sample covariance data.\footnote{An unweighted least squares fit based only on the raw measurements (not the average and covariance data) performed even worse.}





Thanks to this scaling, both the QIP and LS uncertainty zones include all three true system lines at almost every frequency (Fig.~\ref{fig:experiment}.a). As demonstrated in the Fig.~\ref{fig:experiment}.b magnification, both QIP and LS exclude some of the raw data---they have been forced to (approximately) include only the test averages. The large error in the nominal LS model alongside its minuscule original uncertainty in Fig.~\ref{fig:experiment}.b are the cause of the LS model's 10,860-fold uncertainty expansion.

Due to the noise relaxation, the QIP model does not perfectly include the true system models at all frequencies. At the 5 rad/s double pole, we see two of the true system configurations escape the QIP bound between sequential measurements in the magnitude plot. This is partially due to sparse sampling of frequency points. But even violating the containment at a measured frequency is expected given the noisy measurements and the nature of our noise relaxation.

Most importantly, the LS uncertainty set is much larger than that of the QIP. The LS parameter uncertainty estimates are based on the covariance of these parameters if the model were true, linear, and Gaussian, and these parameters were estimated again using new data. These estimates do not lead to a tight, uncertainty-bounding model. QIP does. It is explicitly optimized to use the minimal amount of uncertainty necessary to account for each statistical data point (a mean and sample variance) representing an observed behavior of the system.

\section{Conclusion}\label{discussion}


When people use $\hinf$ control they expect a guarantee of performance, a responsibility which $\hinf$ control delegates to the system model-set. Due to the importance of this guarantee, practitioners will estimate uncertainty which is large enough to make the system work---sacrificing performance. It was our aim to extract the best possible performance from a system, and so we sought leaner, more aggressive model-sets. 

This led us to visualize the model-set as a high dimensional degenerate quadric in the space of inputs and outputs. We introduced the QIP as a lossless convexification for the problem of fitting a minimal quadric around a list of observed data points. 
This new machinery appears to be somewhat more general than our context of identification for robust control, since it offers a geometry-based alternative to the nearly universal least squares problem. 
And within system identification, there are many approaches which use least squares and could potentially identify robust models via QIP.

Our motivation is robotics, where there is often little intuition to be had for the proper shape of uncertainty. We hope to use this technique to imbue footstep planners \cite{KimZhaoThomasFernandezSentis2016TRO} and series elastic robot controllers \cite{RaoThomasEA2017ICRA} with uncertainty-driven performance limits. 


\bibliographystyle{IEEEtran}
\bibliography{lmi,sea,wbc,sysID}

\end{document}